\documentclass[12pt,letterpaper]{article}
\usepackage[utf8]{inputenc}
\usepackage{colortbl,color}
\usepackage{amsmath,amsfonts,amssymb,amsthm,amstext}
\usepackage{graphicx,mathtools,array,float}
\usepackage[affil-it]{authblk}

\newtheorem{theorem}{Theorem}

\newtheorem{lemma}[theorem]{Lemma}
\newtheorem{corollary}[theorem]{Corollary}

\newtheorem{remark}[theorem]{Remark}

\definecolor{olivegreen}{rgb}{0.05,0.7,0.1}

\title{Partitions of complete twisted graphs into plane spanning trees}

\author{Ana Paulina Figueroa} 
\affil{Departamento Acad\'emico de Matem\'aticas, Instituto Tecnol\'ogico Aut\'onomo de M\'exico.}
\author{Eduardo Rivera-Campo}
\affil{Departamento de Matem\'aticas\\Universidad Aut\'onoma Metropolitana - Iztapalapa.}

\date{}

\begin{document}

\maketitle              
\begin{abstract}
We characterize all partitions of the  complete twisted  graph $T_{2n}$ into plane spanning trees. In the case of partitions of $T_{2n}$ into isomorphic  plane spanning trees, we show that all trees in these partitions must be balanced double stars. As a consequence of our results, any complete topological graph with $n$ vertices contains a complete topological subgraph with $m \geq c\log^{1/8} n$ vertices  that admits a partition into plane spanning trees.

\medskip


\end{abstract}

\section{Introduction}
\label{Introducction}

Let $P$ be a set of points in the plane. A \emph{topological graph} with vertex set $P$ is a  graph $G$ where edges are simple continuous arcs drawn in the plane in such a way that any two edges have at most one point in common. Two topological graphs $G$ and $F$ are called \emph{weakly isomorphic} if there is a graph isomorphism between $G$ and $F$ such that two edges of $G$ intersect if and only if the corresponding edges of $F$ do.
A \emph{geometric graph} is a topological graph where all edges are straight line segments. A geometric graph $G$ is \emph{convex} if the vertices of $G$ are in convex position. We denote the complete convex graph with $n$ vertices by $C_n$.

The well-known so called ``Happy Ending" theorem of Erd\H{o}s and Szekeres \cite{ES35,ES60} can be reformulated as follows: every complete geometric graph with $n$ vertices has a complete geometric subgraph which is weakly isomorphic to a complete convex graph $C_m$ with $m \geq c \log n$. Harborth and Mengersen \cite{HM92} introduced a class of topological graphs (later named \emph{twisted graphs}) as examples of topological graphs without topological subgraphs weakly isomorphic to any complete convex geometric graph $C_n$ with $n \geq 5$.

The complete twisted graph $T_n$ is a complete topological graph with $n$ vertices $v_1, v_2, \ldots, v_n$ in which two edges $v_iv_j$ $(i<j)$ and $v_sv_t$ $(s<t)$ cross each other if and only if either $i<s<t<j$ or $s<i<j<t$, see Fig. \ref{T_6}. 

\begin{figure}[ht!]
\centering
  \includegraphics[width= 4in]{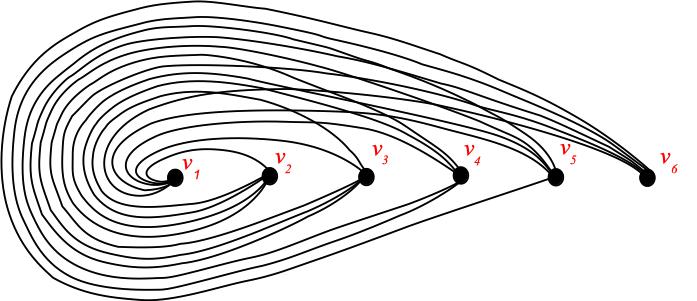}
 \caption{Complete twisted graph $T_6$.}
  \label{T_6}
\end{figure}

Pach et al. \cite{PST03} proved that every complete topological graph with $n$ vertices contains a complete topological subgraph with $m \geq c\log^{1/8} n$ vertices which is weakly isomorphic to either the complete convex graph $C_m$ or to the complete twisted graph $T_m$. Arroyo et al. \cite{ARST19} generalized this result to abstract rotational systems. 

The above result by Pach et al. motivated the study of further properties of twisted graphs; in particular the study of properties shared by convex and twisted graphs. If a given property $P$ is shared by complete convex graphs and complete twisted graphs, then any complete topological graph with $n$ vertices contains a complete topological subgraph with $m \geq c\log^{1/8} n$ vertices  that satisfies property $P$.

Given a topological graph $G$, a subgraph $F$ of $G$ is \emph{plane} or \emph{non self-crossing} if no two edges of $F$ cross in $G$.

Oma\~{n}a-Pulido and  Rivera-Campo \cite{OR12};  \'Abrego et al. \cite{AFFMR22} and  Figueroa and Fres\'an-Figueroa \cite{FF20} studied some problems for which the corresponding versions for the complete convex graph $C_n$ are known.

A \emph{partition} of a complete graph $K_{2n}$ into spanning trees is a collection of $n$ edge disjoint spanning trees $S_1, S_2, \ldots, S_n$ of $K_{2n}$. Since $K_{2n}$ has ${\binom{2n}{2}} = n(2n-1)$ edges and each spanning tree of $K_{2n}$ has $2n-1$ edges, $(E(S_1), E(S_2), \ldots, E(S_n))$ is a partition of the set of edges of $K_{2n}$.

A textbook exercise states that for each positive integer $n$ there is a partition of the complete abstract graph $K_{2n}$ into $n$ spanning paths and a partition of $K_{2n}$ into $n$  spanning balanced double stars. Jaeger et al. \cite{JPK83} gave a generalization of this result concerning partitions of $(2k+1)$-regular graphs with perfect matchings into balanced double stars of order $2(k+1)$. For topological graphs a non self-crossing condition on each tree in the partition is required which changes the problem drastically from that of abstract graphs. For instance, as a corollary of Lemma \ref{primeroiso} in Section \ref{Section 4}, one can see that for $n \geq 3$ there is no partition of the complete twisted graph $T_{2n}$ into $n$ plane spanning paths.

Recently,  the long-standing conjecture that any complete geometric graph on $2n$ vertices can be partitioned into $n$ plane spanning trees was disproved by Aichholzer et al. \cite{AOOPSSTV}.  For complete convex graphs, the existence of such a partition follows from a result by Bernhart and Kainen  \cite{BK79}; Bose et al. \cite{BHRW06} gave a characterization of such partitions; as it turns out, for each partition of $C_n$ into plane spanning trees, all trees in the partition must be pairwise isomorphic. In this paper we characterize partitions of the complete twisted graph $T_{2n}$ into arbitrary plane spanning trees and partitions of $T_{2n}$ into isomorphic plane spanning trees, see Fig. \ref{2particiones}.


\begin{figure}[ht!]
\centering
  \includegraphics[width= 4.5 in]{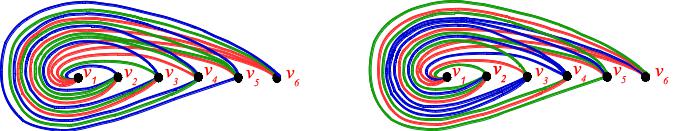}
 \caption{A partition of $T_6$ into non-isomorphic plane trees (left) and a partition of $T_6$ into isomorphic plane trees (right).}
  \label{2particiones}
\end{figure}

Let $T_{2n}$ be a complete twisted graph with vertices $v_1,v_2, \ldots, v_{2n}$. For the purpose of this article it is convenient to rename vertices $v_{n+1},v_{n+2}, \ldots, v_{2n}$ as follows: for $i = 1, 2, \ldots, n$ let $w_i = v_{2n+1-i}$. From now onward the vertices of $T_{2n}$ will be denoted by $v_1, v_2, \ldots, v_n, w_n, w_{n-1}, \ldots, w_1 $ in that order.

Let $n$ be a positive integer. A \emph{spanning balanced double star} of $T_{2n}$ with center edge $v_iw_i$ is a spanning tree of $T_{2n}$ containing edge $v_iw_i$ and such that  $v_i$ is adjacent to half of the remaining vertices and $w_i$ is adjacent to the other half. For example, the red tree in Fig. \ref{2particiones} (left) and all trees in Fig. \ref{2particiones} (right) are spanning balanced double stars of $T_6$.

\section{Preliminary remarks}
\label{Preliminaries}

A spanning tree $S$ of a complete twisted graph $T_{2n}$ is \emph{plane} or \emph{non self-crossing} if no two edges  of $S$ cross in $T_{2n}$.

Let $T_{2n}$ be a complete twisted graph with $2n$ vertices and $\mathcal{S} = \{S_1, S_2, \ldots, \allowbreak S_{n}\}$ be a partition of $T_{2n}$ into plane spanning trees. Since edges $v_1w_1, \allowbreak v_2w_2, \ldots, v_nw_n$ are pairwise crossing in $T_{2n}$, they all must lie in different trees in $\mathcal{S}$.  Without loss of generality, from now onward, we assume edge $v_iw_i$ lies in tree $S_i$ for $i=1, 2, \ldots, n$.

\begin{remark}
\label{S_1}
    In any partition $\mathcal{S} = \{S_1, S_2, \ldots, S_{n}\}$ of the complete twisted graph $T_{2n}$ into plane spanning trees, tree $S_1$ is a balanced double star with center edge $v_1w_1$.
\end{remark}

\begin{proof}
   Vertices $v_2, v_3, \ldots, v_n, w_n, w_{n-1}, \ldots, w_2$ are independent in $S_1$ since all edges joining two of these vertices cross edge $v_1w_1$ which is an edge of $S_1$. Therefore $S_1$ is a double star with center edge $v_1w_1$. 
   
   Clearly $d_{S_i}(v_1) \ge 1$ for $i=2,3, \ldots, n $, therefore vertex $v_1$ cannot be adjacent to more than $n$ vertices in $S_1$. Analogously vertex $w_1$ is adjacent to at most $n$ in $S_1$. This implies $d_{S_1}(v_1) = d_{S_1}(w_1) = n $ and therefore $S_1$ is balanced.
\end{proof}

\begin{remark}
\label{menos}

Let $n \geq 2$ be an integer. If $\mathcal{S} = \{S_1, S_2, \ldots, S_{n}\}$ is a partition of the complete twisted graph $T_{2n}$ into plane spanning trees, then  $S_2 - \{v_1,w_1\},  S_3 - \{v_1,w_1\}, \allowbreak \ldots, S_n - \{v_1,w_1\}$ is a partition of the complete twisted graph with vertices $v_2, v_3, \ldots, v_{n}, w_n, w_{n-1}, \ldots, w_2$ into plane spanning trees.
    
\end{remark}

\begin{proof}
  By Remark \ref{S_1}, $d_{S_1}(v_1) = d_{S_1}(w_1) = n$. This implies that vertices $v_1$ and $w_1$ have degree $1$ in trees $S_2, S_3, \ldots, S_n$. Therefore $S_i - \{v_1,w_1\}$ is a plane tree for $i=2, 3, \ldots, n$. Since $v_2, v_3, \ldots v_n, w_n, w_{n-1}, \ldots, w_2$ are independent in $S_1$, all edges joining two of these vertices  lie in one of the trees $S_2 - \{v_1,w_1\}, S_3 - \{v_1,w_1\}, \ldots, S_n - \{v_1,w_1\}$. 
\end{proof}

\section{Arbitrary partitions}

In this section we study partitions of $T_{2n}$ into arbitrary plane spanning trees.

\begin{lemma}
\label{induced}
Let $n \geq 2$ be an integer. If $\mathcal{S} = \{S_1, S_2, \ldots, S_{n}\}$ is a partition of the complete twisted graph $T_{2n}$ into plane spanning trees, then for $i=1, 2, \ldots, n-1$, the subgraph $S_i[V_i]$ of $S_i$ induced by the set of vertices $V_i = \{v_i, v_{i+1}, \ldots, v_n, w_n,\allowbreak w_{n-1}, \ldots, w_i\}$ is a balanced double star with center edge $v_iw_i$. Moreover, tree $S_i[V_i]$ contains edges $v_iw_{n-1}, v_iw_{n-2}, \ldots, v_iw_{i+1}$, edges $w_iv_{n-1}, w_iv_{n-2}, \allowbreak \ldots,  w_iv_{i+1}$ and one of the following pairs of edges $\{v_iw_n, w_iv_n\}$ or $\{v_iv_{n}, w_iw_n\}$.
\end{lemma}

\begin{proof}
If $n = 2$, then  $\mathcal{S} = \{S_1, S_2\}$. By Remark \ref{S_1}, tree $S_1$ is a balanced double star with center edge $v_1w_1$ and edge set $\{v_1w_2, v_1w_1, w_1v_2\}$  or $\{v_1v_2, v_1w_1, w_1w_2\}$. 

We proceed by induction assuming the result holds for $n = k$. Let $\mathcal{S} = \{S_1, \allowbreak  S_2,  \ldots,  S_{k+1}\}$ be a partition of the complete twisted graph with vertices $v_1, v_2, \ldots, v_{k+1}, \allowbreak w_{k+1},  w_{k}, \ldots, w_1$ into plane spanning trees.  

For $i=2, 3, \ldots, k+1$ let $S'_i = S_i - \{v_1,w_1\}$. By Remark \ref{menos}, $\mathcal{S'} = \{S'_2, S'_3, \allowbreak \ldots, S'_{k+1}\}$ is a partition of the complete twisted graph with vertices $v_2, v_3, \ldots, \allowbreak \ v_{k+1}, w_{k+1}, w_{k},\ldots, w_2$ into plane spanning trees. 

By the induction hypothesis, for $i = 2, 3, \ldots, k$, the subgraph $S'[V_i]$ of $S'_i$ is a balanced double star with center edge $v_iw_i$ that contains edges $v_iw_k, v_iw_{k-1}, \ldots, \allowbreak v_iw_{i+1}$, edges $w_iv_k, w_iv_{k-1}, \ldots, w_iv_{i+1}$ and one of the pairs of edges $\{v_iw_{k+1}, \allowbreak w_iv_{k+1}\}$ or $\{v_iv_{k+1},  w_iw_{k+1}\}$. Notice that $S'[V_i] = S[V_i]$ for $i = 2, 3, \ldots, k$. 

Finally, by Remark \ref{S_1}, tree $S_1$ is a balanced double star with center edge $v_1w_1$. Therefore each vertex $v_2, v_3, \ldots, v_{k+1}, w_{k+1}, w_k, \ldots, w_2$ is adjacent to either $v_1$ or $w_1$. Let $j$ be an integer with $2 \leq j \leq k$. Notice that for $i = 2, 3, \ldots, k $, edge $v_1w_j$ crosses edges $v_iw_{k+1}$ and $v_iv_{k+1}$ of which one is an edge of $S_i$, therefore for $i = 2, 3, \ldots, k$, edge $v_1w_j$ is not an edge of $S_i$. Since $v_1w_j$ also crosses edge $v_{k+1}w_{k+1}$ which is an edge of $S_{k+1}$, edge $v_1w_j$ is not an edge of $S_{k+1}$ either. Therefore for $j = 2, 3, \ldots, k$, edge $v_1w_j$ must be an edge of $S_1$. Analogously $w_1v_j \in E(S_1)$ for $j = 2, 3, \ldots, k$. 

Since $S_1$ is a balanced double star with center edge $v_1w_1$, either $v_1w_{k+1}, \allowbreak w_1v_{k+1} \in E(S_1)$ or $v_1v_{k+1}, w_1 w_{k+1} \in E(S_1)$. Thus the result also holds for $n=k+1$. By induction it holds for all integers $n \geq 2$.

\end{proof}

\begin{remark}
\label{edgesv_1v_i}
Let $n \geq 2$ be an integer. If $\mathcal{S} = \{S_1, S_2, \ldots, S_{n}\}$ is a partition of the complete twisted graph $T_{2n}$ into plane spanning trees, then edges $v_1v_2, v_1v_3, \allowbreak \ldots, v_1v_n, v_1w_n$ lie in different trees and edges $w_1w_2, w_1w_3,\allowbreak \ldots, w_1w_n, w_1v_n$ also lie in different trees.

\end{remark}

\begin{proof}

By Lemma \ref{induced}, $S_1$ is a balanced double star with center edge $v_1w_1$ that contains edges $w_1v_2, w_1v_3, \ldots, w_1v_{n-1}$ and edges $v_1w_2, v_1w_3, \ldots, v_1w_{n-1}$. Notice that in both cases edges $v_1v_2, v_1v_3, \ldots, v_1v_{n-1}$ and $w_1w_2, w_1w_3, \ldots, w_1w_{n-1}$ are not edges of $S_1$; since $\mathcal{S}$ is a partition of $T_{2n}$, all of these edges must lie in some tree $S_i$ with $i \ge 2 $.  

Also by Lemma \ref{induced}, tree $S_1$ contains  one of the following pairs of edges $\{v_1w_n, w_1v_n\}$ or $\{v_1v_n, w_1w_n\}$. In the first case $v_1w_n, w_1v_n \in E(S_1)$, $v_1v_n \in E(S_i)$  for some $i \geq 2$ and  $w_1w_n \in E(S_j)$  for some $j \geq 2$ and, in the second case $v_1v_n, w_1w_n  \in E(S_1)$, $v_1w_n \in E(S_r)$  for some $r \geq 2$ and $w_1v_n \in E(S_t)$  for some $t \geq 2$. 

From the proof of Remark \ref{menos}, $d_{S_i}(v_1) = d_{S_i}(w_1) = 1$ for $i=2, 3, \ldots, n$. Therefore edges $v_1v_2, v_1v_3, \ldots, v_1v_{n}, v_1w_n$ lie in different trees and $w_1w_2, w_1w_3, \allowbreak   \ldots, w_1w_{n},w_1v_n$ lie in different trees.
\end{proof}

Next we give a recursive formula for the number of partitions of $T_{2n}$ into arbitrary plane spanning trees.

\begin{theorem}
\label{recursiva} 

The number $p_n$ of partitions of the complete twisted graph $T_{2n}$ into arbitrary plane spanning trees satisfies the recursive formula $p_{n+1} = 5(2^{2(n-2)})p_n$ for $n\geq2$.

\end{theorem}

\begin{proof}

 Let $n \geq 2$ be an integer and $\mathcal{S'} = \{S'_2, S'_3, \ldots, S'_{n+1}\}$ be a partition of the complete twisted graph with vertices $v_2, v_3, \ldots, v_{n+1}, w_{n+1}, \allowbreak w_n, \ldots, w_2$ into plane spanning trees. We show that there are $5(2^{2(n-2)})$ partitions $\mathcal{S} = \{S_1, S_2, \allowbreak \ldots, S_{n+1}\}$ of $T_{2(n+1)}$ where $S_i - \{v_1,w_1\} = S'_i$ for $i=2, 3, \ldots, n+1$. By Lemma \ref{induced}, for all such partitions, tree $S_1$ is a balanced double star with center edge $v_1w_1$, containing all edges $v_1w_i$ and $w_1v_i$ for $i = 2, 3, \ldots, n$ and one of the following pairs of edges $\{v_1w_{n+1}, \allowbreak w_1v_{n+1}\}$ or $\{v_1v_{n+1}, \allowbreak w_1w_{n+1}\}$. 
 
 We show first that there are $2^{2(n-1)}$ partitions where tree $S_1$ contains edges $v_1w_{n+1}$ and $ w_1v_{n+1}$. Notice that in such cases the edges of $T_{2(n+1)}$ not contained in $E(S_1) \cup E(S'_2) \cup \cdots \cup E(S'_{n+1})$ are edges $v_1v_i$ and $w_1w_i$ with $i = 2, 3, \ldots, n+1$. We claim that there are $2^{n-1}$ ways to attach  edges $v_1v_i$, with $i = 2, 3, \ldots, n+1$, and $2^{n-1}$ ways to attach edges $w_1w_i$, with $i = 2, 3, \ldots, n+1$, to $S'_2, S'_3, \ldots, S'_{n+1}$ to obtain a partition $S_1, S_2, \ldots, S_{n+1}$ of $T_{2(n+1)}$.
 
 By Remark \ref{edgesv_1v_i}, each tree in $\mathcal{S'}$ contains one of the edges $v_2v_3, v_2v_4, \ldots, v_2v_{n+1}, \allowbreak v_2w_{n+1}$. For $j = 3, 4, \ldots, n+1$, let $i_j$ be an integer such that $v_2v_{j}$ is an edge of $S'_{i_j}$ and let $i_{n+2}$ be an integer such that edge $v_2w_{n+1}$ lies in $S'_{i_{n+2}}$. 
 
 Notice that for $j=3, 4, \ldots, n+1$ any single edge in $\{v_1v_2, v_1v_3, \ldots, v_1v_{j}\}$ can be attached to $S'_{i_j}$ and any single edge in $\{v_1v_2, v_1v_3, \ldots, v_1v_{n+1}\}$ can be attached to $S'_{i_{n+2}}$ obtaining a plane tree.  We can choose first any edge in $\{v_1v_2, v_1v_3\}$ to attach to $S'_{i_3}$, then choose any remaining edge in $\{v_1v_2, v_1v_3, \allowbreak v_1v_4\}$ to attach to $S'_{i_4}$ and so on. For $S'_{i_3}, S'_{i_4}, \ldots, S'_{i_{n+1}}$ there are two remaining edges and for $S'_{i_{n+2}}$ only one possible edge remains. This gives $2^{n-1}$ ways to attach edges $v_1v_i$, with $i = 2, 3, \ldots, n+1$ to $S'_2, S'_3, \ldots, S'_{n+1}$. 
 
 Analogously, there are also $2^{n-1}$ ways to attach edges $w_1w_i$, with $i = 2, 3, \allowbreak \ldots, n+1$ to $S'_2, S'_3, \ldots, S'_{n+1}$. Since the ways of attaching edges $v_1v_i$ and edges $w_1w_i$, with $i = 2, 3, \ldots, n+1$ are independent, there are $2^{2(n-1)}$ ways to obtain a partition $\{S_1, S_2, \ldots, S_{n+1}\}$ of $T_{2(n+1)}$ from $\{S'_2, S'_3, \ldots, S'_{n+1}\}$ where $S_1$ contains edges $v_1w_{n+1}, w_1v_{n+1}$  and $S_i - \{v_1,w_1\} = S'_i$ for $i=2, 3, \ldots, n+1$.

 With a similar argument we can prove that there are $2^{2(n-2)}$ ways to obtain a partition $\{S_1, S_2, \ldots, S_{n+1}\}$ of $T_{2(n+1)}$ from $\{S'_2, S'_3, \ldots, S'_{n+1}\}$ where $S_1$ contains edges $v_1v_{n+1}$ and $ w_1w_{n+1}$  and $S_i - \{v_1,w_1\} = S'_i$ for $i=2, 3, \ldots, n+1$. Therefore $p_{n+1} = (2^{2(n-1)} + 2^{2(n-2)})p_n = 5(2^{2(n-2)})p_n$. 
\end{proof}

\begin{corollary}
\label{principalnoiso}

For  $n \geq 2$ there are $(5^{n-2})(2^{1+(n-2)(n-3)})$ partitions of the complete twisted graph $T_{2n}$ into plane spanning trees.

\end{corollary}

\begin{proof}
By Theorem \ref{recursiva}, the number $p_n$ of partitions of the complete twisted graph $T_{2n}$ into plane spanning trees satisfies the recursive formula $p_{n+1} =  5(2^{2(n-2)})p_n$. 

A simple proof by induction shows that for $p_n = (5^{n-2})(2^{1+(n-2)(n-3)})$ for $n \geq 2$, since $T_4$ has $p_2 = 2$ such partitions.
\end{proof}

\section{Partitions into isomorphic trees}
\label{Section 4}

In this section we give a formula for the number of partitions of the complete twisted graph $T_{2n}$ into isomorphic plane spanning trees $S_1, S_2, \ldots, S_n$.

\begin{lemma}
\label{primeroiso}
Let $n\ge 2$ be an integer and $\mathcal{S} = \{S_1, S_2, \ldots, S_n\}$ be a partition  of the complete twisted graph $T_{2n}$ into isomorphic plane spanning trees. Then for $i=1, 2, \ldots,n$ tree $S_i$ is a balanced double star with center edge $v_iw_i$. 
\end{lemma}

\begin{proof}
    By Remark \ref{S_1}, tree $S_1$ is a balanced double star with center edge $v_1v_{2n}$. Since $S_i$ is isomorphic to $S_1$ for $i = 2, 3, \ldots, n$, all trees in $\mathcal{S}$ are balanced double stars. By Lemma \ref{induced}, for $i = 2, 3, \ldots, n-1$, tree $S_i[V_i]$ is a balanced double star with center edge $v_iw_i$, therefore $v_iw_i$ is also the center edge of $S_i$  since $S_i[V_i]]$ is a balanced double star contained in balanced double star $S_i$. 
    
    Finally, since we have shown that for $i = 1, 2, \ldots n-1 $ tree $S_i$ is a balanced double star with center edge $v_iw_i$, we have $d_{S_i}(v_n) = d_{S_i}(w_n) = 1 $ for  $i = 1, 2, \ldots n-1 $. This implies $d_{S_n}(v_n) = d_{S_n}(w_{n}) = n$ and therefore edge $v_nw_n$ must be the center edge of $S_n$.  
\end{proof}

For any vertex $u$ of a graph $G$, let $N_G(u)$ be the set of vertices of $G$ which are adjacent to $u$ in $G$.

Let $n \geq 2$ be an integer. For $k=1, 2, \ldots, n$, denote by $A_k$ be the spanning balanced double star of $T_{2n}$ with center edge $v_nw_n$ and such that $N_{A_k}(v_n) =\{v_1,v_2, \ldots, v_ {n-k}\}\cup \{w_{n}, \allowbreak w_{n-1}, \ldots, w_{n-(k-1)} \}$ for $k \leq n-1$ and $N_{A_k}(v_n) = \{w_n, \allowbreak w_{n-1}, \ldots, \allowbreak w_1\}$ for $k=n$. Since $A_k$ is a balanced double star with center edge $v_nw_n$, $N_{A_k}(w_n) =\{w_1,w_2, \ldots, \allowbreak w_ {n-k}\}  \cup \{v_{n},  v_{n-1}, \ldots, v_{n-(k-1)} \}$ for $k \leq n-1$ and $N_{A_k}(w_n) = \{v_n, v_{n-1}, \ldots, \allowbreak v_1\}$ for $k=n$.

\begin{theorem}
\label{Sn=Ak}
  Let $n \ge 2$ be an integer. If $\mathcal{S} = \{S_1, S_2, \ldots, S_n\}$ is a partition  of the complete twisted graph $T_{2n}$ into isomorphic plane spanning trees, then there is an integer $k$ with  $1 \leq k \leq n$  such that $S_n =A_{k}$.
\end{theorem}

\begin{proof}
 By Lemma \ref{primeroiso}, tree $S_n$ is a balanced double star with center edge $v_nw_n$; let $k$ $=$ max $\{j : v_nw_{n}, v_nw_{n-1}, \ldots, v_nw_{n-j+1} \in E(S_n)\}$. 
 
 If $k = n$, then vertex $v_{n}$ is adjacent in $S_n$ to all vertices $w_n, w_{n-1}, \ldots, w_1$. In this case $S_n = A_n$ since $S_n$ is a balanced double star with center edge $v_nw_n$.
 
 If $k < n$, then $w_nw_{n-k} \in E(S_n)$; this implies $w_nw_{n-k-1}, w_nw_{n-k-2},  \dots, \allowbreak w_nw_1 \in E(S_n)$, otherwise $v_nw_{n-k-j}$ is an edge of $S_n$ for some integer $j$ with $1\leq j \leq n-k-1$, which is not possible since edges $w_nw_{n-k}$ and $v_nw_{n-k-j}$ cross each other. Therefore $S_n = A_k$.
\end{proof}

\begin{lemma} 
\label{exterior}

For $i = 2, 3, \ldots, n - 1$, tree $S_i$ contains edges $v_1v_i, v_2v_i, \ldots, v_{i-1}v_i$ and edges $w_1w_i, w_2w_i, \ldots, w_{i-1}w_i$. 
\end{lemma}

\begin{proof}

Let $i = 2, 3, \ldots, n - 1$ and $t = 1, 2, \ldots, i-1$. As $S_i$ is a balanced double star with center edge $v_iw_i$, vertex $v_t$ is adjacent either to $v_i$ or to $w_i$ in $S_i$. By Lemma \ref{induced}, tree $S_i$ contains one of the edges $v_iv_n$ or $v_iw_n$. Notice that edge $v_tw_i$ is not an edge of $S_i$ as it crosses both edges $v_iv_n$ and $v_iw_n$ in $T_{2n}$. This implies $v_t$ is adjacent to $v_i$ in $S_i$. Analogously $w_t$ is adjacent to $w_i$ in $S_i$.  

\end{proof}

\begin{theorem}
\label{determinado}
    Let $n \ge 2$ be an integer. For $k=1, 2, \ldots, n$ there exists a unique partition $\mathcal{S} = \{S_1, S_2, \ldots, \allowbreak  S_n\}$  of the complete twisted graph $T_{2n}$ into isomorphic plane spanning trees such that $S_n=A_k$.
\end{theorem}
\begin{proof}  

There are two partitions of $T_4$ into isomorphic plane spanning trees, one contains $A_1$ and the other contains $A_2$, see Fig. \ref{T_4}.

\begin{figure}[ht!]
\centering
  \includegraphics[width= 4 in]{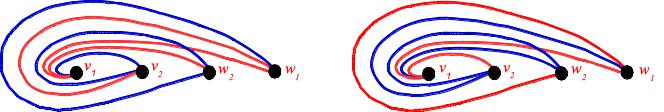}
  \caption{Two partitions of $T_4$. $S_2=A_1$ on the left and $S_2=A_2$ on the right. }
  \label{T_4}
\end{figure}

Assume $n \geq 3$. By Lemma \ref{induced} and Lemma \ref{exterior}, for $i=1, 2, \ldots, n-1$ the only edges of $S_i$ that are not fixed in any partition $\mathcal{S} = \{S_1, S_2, \ldots, S_n\}$ of $T_{2n}$ into isomorphic plane spanning trees are those incident to $v_n$ and to $w_n$. Let $k = 1, 2, \ldots, n$ and let $\mathcal{S} = \{S_1, S_2, \ldots, S_n\}$ be a such a partition of $T_{2n}$ with $S_n = A_k$. Notice that since each tree $S_i$ is a balanced double star with center edge $v_iw_i$, in order to prove that tree $S_i$ is  uniquely determined by $A_k$, it suffices to show which of the vertices $v_i$ or $w_i$ is adjacent to $v_n$ in $S_i$ for $i= 1, 2, \ldots, n-1$.

We first look at the case $k=n$. Then $S_n = A_n$ contains all edges $v_nw_i$ with $i = 1, 2, \ldots, n-1$. In this case $v_i$ must be adjacent to $v_n$ in $S_i$ for $i=1, 2, \ldots n-1$. 

Assume now $k < n$.  If $1 \leq i \leq n-k$, then $v_nv_i$ is an edge of $A_k=S_n$ and therefore cannot be an edge of $S_i$. This implies that $w_i$  is adjacent to $v_n$ in $S_i$ for $1 \leq i \leq n-k$. See tree $S_1$ in Fig. \ref{S_3=A_2}.

Analogously,  if $i > n-k$, then $w_nv_i$ is an edge of $S_n$ and therefore cannot be an edge of $S_i$. This implies that $v_i$  is adjacent to $v_n$ in $S_i$ for $i > n-k$. See tree $S_2$ in Fig. \ref{S_3=A_2}.

\begin{figure}[ht!]
\centering
  \includegraphics[width= 4.5 in]{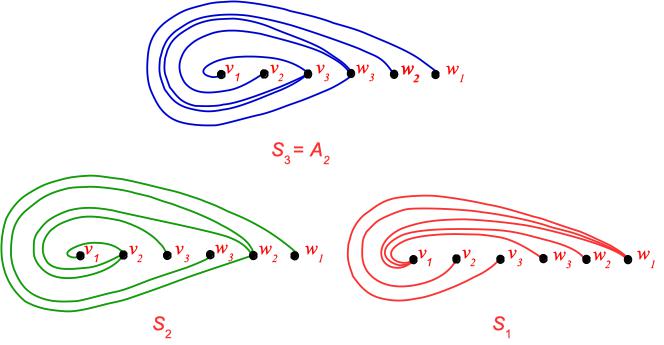}
  \caption{Partition of $T_6$ into balanced double stars $S_1,S_2, S_3$ with $S_3=A_2$. Here $v_3v_1 \in E(S_3)$ implies $v_3w_1 \in E(S_1)$ and $w_3v_2 \in E(S_3)$ implies $v_3v_2 \in E(S_2)$.}
  \label{S_3=A_2}
\end{figure}

\end{proof}

As an immediate consequence of theorems \ref{Sn=Ak} and \ref{determinado}, we have the following. 

\begin{corollary}
Let $n\geq 2$ be an integer. There are $n$ partitions of $T_{2n}$  into isomorphic plane spanning trees. 
\end{corollary}


\begin{thebibliography}{99}

\bibitem{AFFMR22}
\'Abrego, B. M., Fern\'andez-Merchant, S., Figueroa, A. P., Montellano-Ballesteros, J. J., Rivera-Campo, E., The Crossing Number of Twisted Graphs, Graphs and Combinatorics, Volume 38 (2022), Article number: 134.

\bibitem{AOOPSSTV} 
Aichholzer, O.,  Obenaus, J.,    Orthaber, J., Paul, R.,  Schnider, P. ,  Steiner, R.,  Taubner, T.,  Vogtenhuber, V.,
Edge Partitions of Complete Geometric Graphs,
38th International Symposium on Computational Geometry (SoCG 2022), 6:1 -- 6:16. 

\bibitem{ARST19}
Arroyo, A., Richter, R. B., Salazar, G., Sullivan, M., The unavoidable rotation systems, 	arXiv:1910.12834.

\bibitem{BK79}
Bernhart, F.   Kainen, P. C.,
The book thickness of a graph,
Journal of Combinatorial Theory, Series B,
Volume 27, Issue 3 (1979), 320 -- 331.

\bibitem{BHRW06}
Bose, P., Hurtado, F.,  Rivera-Campo, E.,  Wood, D. R., Partitions
of complete geometric graphs into plane trees. Comput. Geom., 34(2) (2006) 116 -- 125.

\bibitem{ES35}
Erd\H{o}s, P., Szekeres, G., A combinatorial problem in geometry, Compositio Mathematica 2, (1935), 463 -- 470.

\bibitem{ES60}
Erd\H{o}s, P., Szekeres, G., On some extremum problems in elementary geometry. Ann. Univ. Sci. Bp. Rolan do E\H{o}tv\H{o}s Nomin., Sect. Math. III-IV (1960–1961), 53 -- 62.

\bibitem{FF20}
Figueroa, A. P., Fres\'an-Figueroa, J., The biplanar tree graph. Bol. Soc. Mat. Mex. Volume 26 (2020), 795 -- 806.

\bibitem{HM92}
Harborth, H., Mengersen, I., Drawings of the
complete graph with maximum number of crossings,
Congr. Numer., 88 (1992), 225 -- 228.

\bibitem{JPK83}
Jaeger, F., Payan, C., Kouider, M., Partition of odd regular graphs into bistars, Discrete Math 46 (1983) 93 -- 94.

\bibitem{OR12}
Oma\~{n}a-Pulido, E., Rivera-Campo, E., Notes in twisted graphs, in Márquez, Alberto (ed.) et al., Computational Geometry, Lecture Notes in Computer Science 7579, (2012), 119 -- 125.


\bibitem{PST03}
Pach, J., Solymosi, J., T\'oth, G., Unavoidable configurations in complete topological graphs, Discrete Comput. Geom. 30 (2003), 311 -- 320.


\end{thebibliography}
\end{document}